\documentclass[12pt]{article}
\usepackage{latexsym,amsmath,theorem,graphicx,amssymb,a4wide}
\usepackage[utf8]{inputenc}
\usepackage{hyperref}

\numberwithin{equation}{section}

\newtheorem{teo}{Theorem}[section]
\newtheorem{lemma}[teo]{Lemma}
\newtheorem{pro}[teo]{Proposition}

{\theorembodyfont{\rmfamily}\newtheorem{re}[teo]{Remark}}
{\theorembodyfont{\rmfamily}}

\DeclareMathOperator{\EE}{\mathbb{E}}
\DeclareMathOperator{\PP}{\mathbb{P}}

\newcommand{\ind}{\mathbf{1}}

\newcommand{\real}{\mathbb{R}}

\newcommand{\integer}{\mathbb{Z}}
\renewcommand{\natural}{\mathbb{N}}

\newcommand{\iid}{\emph{i.i.d. }}
\newcommand{\eg}{\emph{e.g. }}

\renewcommand{\epsilon}{\varepsilon}

 \title{\bf Mixing time for the Repeated Balls into Bins dynamics} 

\author{Nicoletta Cancrini\footnote{nicoletta.cancrini@univaq.it, DIIIE Università dell'Aquila, L'Aquila Italy.}
  \ and Gustavo Posta\footnote{gustavo.posta@uniroma1.it, Dipartimento di Matematica ``G. Castelnuovo'', Università di Roma ``la Sapienza'', Roma Italy.}
}

\begin{document}

\date{}

\noindent

\maketitle
\thispagestyle{empty}

\begin{abstract}
  We consider a nonreversible finite Markov chain called Repeated Balls-into-Bins (RBB) process.
  This process is a discrete time conservative interacting particle system with parallel updates.
  Place initially in $L$ bins $rL$ balls, where $r$ is a fixed positive constant.
  At each time step a ball is removed from each non-empty bin.
  Then \emph{all} these removed balls are uniformly reassigned into bins.
  We prove that the mixing time of the RBB process is of order $L$.
  Furthermore we show that if the initial configuration has $o(L)$ balls per site the equilibrium is attained in $o(L)$ steps.
\end{abstract}

\section{Introduction}
\label{sec:intro}

Consider $L\in\natural$ bins where $rL\in\natural$ balls are initially placed.
At each discrete time step a ball is taken from each non-empty bin and all the balls are uniformly reassigned into bins.
The occupation numbers of balls in bins is an ergodic discrete time finite Markov chain, called the \emph{Repeated Balls-into-Bins process}.

The RBB process arises naturally in different contexts. 
For example the balls in every bin can model customers in a queue, which are served at discrete times.
Each served customer is then reassigned to a random queue.
In this setting the RBB process is a discrete time \emph{closed Jackson network} \cite{Ja,Ke}.
In the algorithmic context the balls are tasks (or \emph{tokens}) in a network of parallel CPU which are reassigned at every round.
See for example \cite{Be:Cl:Na:Pa:Po} for a deeper discussion.

In \cite{Ca:Po} and \cite{Ca:Po:1} we proved the propagation of chaos of the RBB process and studied some equilibrium properties of the limiting nonlinear process.
This system is a conservative interacting particle system in discrete time with \emph{parallel updates}.
We will thus call the balls \emph{particles} and the bins \emph{sites}.
We can think of the RBB process as a zero-range process \cite{Sp} on the complete graph with constant jump rates and parallel updates.
However, because of the parallel updating, it is not reversible.
For this reason its invariant measure is still unknown and the standard techniques to study the convergence rate to equilibrium cannot be used.
The main result of this paper is Theorem~\ref{teo:mai} which implies that the mixing time of the RBB process is of order $L$.
This estimate is needed to generate approximate samples of the invariant measure of the RBB process and  calculate the interesting statistical quantities of the process at equilibrium.
The mixing time of systems with parallel updating rules has been studied for \emph{Probabilistic Cellular Automata}, see for example \cite{Ji:Le} and \cite{Lo}.
However PCA models are reversible, non conservative, and their invariant measure is usually known.

More precisely Theorem~\ref{teo:mai} states that if the system starts in a configuration with $o(L)$ particles per site the equilibrium is attained after $o(L)$ steps, while if the system starts in a configuration with $O(L)$ particles per site the equilibrium is attained after $O(L)$ steps.
This suggests that, as in the zero-range process with constant jump rates (see \cite{He:Sa}), the system separates into a slowly evolving phase of sites with $O(L)$ particles and a quickly evolving phase of sites with $o(L)$ particles.
Thus the slowly evolving phase dissolves into the quickly evolving phase and the mixing time is essentially the time in which it is completely dissolved. 

To prove Theorem~\ref{teo:mai} we use the path coupling technique of Bubley and Dyer \cite{Bu:Dy}.
This technique has been successfully applied in \cite{He:Sa} to estimate the mixing time of the mean field zero range process, which is similar to the RBB process but has \emph{sequential} updating and a known reversible stationary measure.
A remarkable feature of the method used in \cite{He:Sa} is that it does not require reversibility or explicit knowledge of the invariant measure of the model.
We use the same approach as \cite{He:Sa} although the parallel updating of the RBB process requires new ideas.

We briefly outline the strategy of the proof of Theorem~\ref{teo:mai}.

We show that, after a thermalization time depending linearly on the maximum occupation number of the initial state, the following happens.
First the distribution of the site occupation number decays exponentially.
Second there is a coupling such that the distributions of two copies of the RBB process started from two different configurations are close in total variation distance.
In particular we show that this distance can be estimated in terms of the coalescing time of two RBB processes starting from configurations which differ only for one particle.

The paper is organized as follows.
In Section~\ref{sec:cbp} we state the notations and Theorem~\ref{teo:mai}, in Section~\ref{sec:pro} we give its proof.
Finally in the last two sections we prove the two lemmata on which the proof of the main result relies.

\section{Notations and main result}
\label{sec:cbp}

We denote by $\integer_+$ the set of non-negative integers, $\natural:=\integer_+\setminus\{0\}$ and for $L\in\natural$ the \emph{configuration space} $\Omega:=\integer_+^L$.
For any denumerable set $S$ its cardinality, finite or infinite, is denoted by $|S|$ and for any $n\in\natural$ we define $[n]:=\{1,\dots,n\}$.
Given $\eta\in\Omega$ let $\Vert\eta\Vert_\infty:=\max_{x\in[L]}\eta_x$.
If $\mu$ and $\nu$ are two probability measures we denote by $\Vert\mu-\nu\Vert$ their total variation distance.

To keep notation simple, in the following the term \emph{constant} means a number which may depend  only on $r$, where $rL$ is the fixed number of particles.

We introduce an \emph{explicit construction} for the RBB process which will be useful in the sequel.
If $t\in\natural$ define $U(t):=(U_1(t),\dots,U_L(t))$, where $U_1(t),\dots,U_L(t)$ are \iid random variables uniformly distributed on $[L]$ and such that $U(1),U(2),\dots$ are independent.  
The RBB process is a discrete time irreducible finite Markov chain $(\eta(t))_{t\geq0}$ with values in $\Omega$ and invariant measure $\nu$.
For any initial condition $\eta(0)=\eta\in\Omega$, $t\in\integer_+$ and $x\in[L]$ we define recursively
\begin{equation}
  \label{eq:gc}
  \eta_x(t+1)
  :=\eta_x(t)-\ind_{\{\eta_x(t)>0\}}+\sum_{y=1}^L\ind_{\{\eta_y(t)>0\}}\ind_{\{U_y(t+1)=x\}},
\end{equation}
and
\begin{equation}
  \label{eq:bw}
  B_x(t+1)
  :=\sum_{y=1}^L\ind_{\{\eta_y(t)>0\}}\ind_{\{U_y(t+1)=x\}},
  \qquad
  \bar w(t)
  :=\frac{1}{L}\sum_{x=1}^L\ind_{\{\eta_x(t)>0\}}.
\end{equation}
Note that, conditional on $\eta(t)$, the random vector $B(t+1):=(B_1(t+1),\dots,B_L(t+1))$
  has Maxwell-Boltzmann distribution (see for example \cite{Ca:Po:1} \S 3.2) with $L\bar w(t)$ particles and $L$ sites.
Equation \eqref{eq:gc} is equivalent to
\begin{equation}
  \label{eq:gcb}
  \eta_x(t+1)
  =\eta_x(t)-\ind_{\{\eta_x(t)>0\}}+B_x(t+1).
\end{equation}
To keep notation simple we use the standard convention (see \eg \cite{Li}) to denote the initial state of the processes with the same letter of the process, namely
\begin{equation*}
  \PP_\eta(\eta(t)=\xi):=\PP(\eta(t)=\xi\vert \eta(0)=\eta ),
\end{equation*}
for any $\eta,\xi\in\Omega$.
We can now state our main result.

\begin{teo}
  \label{teo:mai}
  Let $(\eta(t))_{t\geq0}$ be the RBB process.
  Then there is a positive constant $c$ such that for any $\epsilon\in(0,1/2)$ and $\eta\in\Omega$ the \emph{(configuration) mixing times}
  \begin{equation*}
    t_{\mathrm{mix}}(\eta,\epsilon)
    :=\inf\big\{t\geq0\colon \Vert \PP_\eta(\eta(t)\in\cdot)-\nu\Vert<\epsilon\big\}
  \end{equation*}
  satisfy
  \begin{equation*}
    t_{\mathrm{mix}}(\eta,\epsilon)\leq c\big(\Vert\eta\Vert_\infty+(\log L)^c\big),
  \end{equation*}
  for every $L\in\natural$ such that $L\geq c/\epsilon$.
\end{teo}
\begin{re}
  This result implies the correct bound on the \emph{mixing time}
  \begin{equation*}
    t_{\mathrm{mix}}(\epsilon):=\sup_{\eta\in\Omega} t_{\mathrm{mix}}(\eta,\epsilon).
  \end{equation*}
  In fact by the \emph{diameter bound} (see \S 7.1.2  in \cite{Le:Pe}) we know that $t_{\mathrm{mix}}(\epsilon)\geq rL/2$.
  By Theorem~\ref{teo:mai} and the bound $\Vert\eta\Vert_\infty\leq r L$ we have that $t_{\mathrm{mix}}(\epsilon)\leq c'L$ for some positive constant  $c'$.
Thus $t_{\mathrm{mix}}(\epsilon)$ is of order $L$.

Note that Theorem~\ref{teo:mai} says more.
If $\Vert\eta\Vert_\infty=o(L)$ then $t_{\mathrm{mix}}(\eta,\epsilon)=o(L)$.
This means that when the system starts in a state  with $o(L)$ particles per site the equilibrium is attained in a time negligible with respect to $t_{\mathrm{mix}}(\epsilon)$.  
\end{re}

\section{Proof of the main result}
\label{sec:pro}

The proof of Theorem~\ref{teo:mai} is based on Lemmata \ref{pro:exp} and \ref{lem:tel}.
The first one states that,  after a thermalization time depending linearly on the initial state, the distribution of the site occupation number of the RBB process decays exponentially.

\begin{lemma}
  \label{pro:exp}
  There are positive constants $\theta$, $\kappa$, $\alpha$ such that 
  \begin{equation}
        \label{eq:ex}
    \EE_\eta\big(e^{\theta\eta_x(t)}\big)
    \leq \kappa\big(1+e^{\theta(\eta_x-\alpha t)}\big),
  \end{equation}
  for all $\eta\in\Omega$, $x\in[L]$ and $t\geq0$.
  In particular, for any $a\geq0$
  \begin{equation}
    \label{eq:Ce}
    \PP_\eta\big(\eta_x(t)\geq(\eta_x-\alpha t)\vee0+a\big)
    \leq 2\kappa e^{-\theta a}.
  \end{equation}
\end{lemma}
The second lemma asserts that the distributions of two RBB processes started from two different configurations, again after a thermalization time depending linearly on them, are close in total variation distance.

\begin{lemma}
  \label{lem:tel}
  There exists a positive constant $c$ such that
  \begin{equation*}
    \Vert\PP_\eta(\eta(t)\in\cdot)-\PP_\xi(\eta(t)\in\cdot)\Vert
    \leq\frac{c}{L}
  \end{equation*}
  for any $\eta,\xi\in\Omega$ and $t>c\big(\Vert\eta\Vert_\infty\vee \Vert\xi\Vert_\infty\vee(\log L)^c\big)$.
\end{lemma}
We can now prove Theorem~\ref{teo:mai}.

\begin{proofof}{Theorem~\ref{teo:mai}}
  We denote by $P_\eta^t:=\PP(\eta(t)\in\cdot)$ the distribution of $\eta(t)$ when $\eta(0)=\eta$, then (see \eg \cite{Le:Pe} Proposition~4.2)
  \begin{equation}
        \label{eq:4}
    \begin{split}
         2\Vert P_\eta^t-\nu\Vert
    &=\sum_{\zeta}\Big\vert P^t_\eta(\{\zeta\})-\sum_{\xi}\nu(\{\xi\}) P^t_\xi(\{\zeta\})\Big\vert\\
     &\leq \sum_{\zeta}\sum_{\xi}\nu(\{\xi\})\big\vert P^t_\eta(\{\zeta\})- P^t_\xi(\{\zeta\})\big\vert
    =    2 \sum_{\xi} \nu(\{\xi\}) \Vert P_\eta^t-P_\xi^t\Vert.
    \end{split}
  \end{equation}
  By Lemma~\ref{lem:tel} and equation~\eqref{eq:4} we have
  \begin{equation}
    \label{eq:5}
    \Vert P_\eta^t-\nu\Vert
    \leq \frac{c}{L}+ \nu\big(\big\{\xi\in\Omega\colon \Vert\xi\Vert_\infty>(\log L)^c\big\}\big),
  \end{equation}
  for any $t\geq c\big( \Vert\eta\Vert_\infty\vee (\log L)^c\big)$.
  Using sub-additivity, ergodicity of the RBB process and Lemma~\ref{pro:exp},  the last term of equation~\eqref{eq:5} can be bounded by
  \begin{equation*}
    \begin{split}
    &\sum_{x=1}^L\nu\big(\big\{\xi\in\Omega\colon \xi_x>(\log L)^c\big\}\big)
    =\lim_{t\to+\infty}\sum_{x=1}^L\PP_\eta\big(\eta_x(t)>(\log L)^c\big)\\
    &=\lim_{t\to+\infty}\sum_{x=1}^L\PP_\eta\big(\eta_x(t)>(\eta_x-\alpha t)\vee 0+(\log L)^c\big)
    \leq 2L\kappa e^{-\theta (\log L)^c},
    \end{split}
  \end{equation*}
  which, taking $c\geq 2$, is smaller than a positive constant times $L^{-2}$ and the result follows.
\end{proofof}

The proofs of Lemmata \ref{pro:exp} and \ref{lem:tel} will be discussed in the next two sections.

\section{Exponential decay of the distribution of the site occupation number}
\label{sec:p1}

To prove Lemma \ref{pro:exp}  we need some preliminary results.
The first one states that the RBB process at time $t$ is stochastically dominated by a Maxwell-Boltzmann distribution with  $tL$ particles and $L$ sites.
This gives a bound on the number of particles per site of the RBB process and it will be crucial because the Maxwell-Boltzmann distribution is \emph{negatively associated} (see \eg \cite{Du:Ra}).

\begin{lemma}
  \label{lem:dom}
  The RBB process is monotone.
  Furthermore there exists a random vector $\tilde{B}(t)$ with Maxwell-Boltzmann distribution with  $tL$ particles and $L$ sites such that:
  \begin{equation}
    \label{eq:bd}
    \eta_x(t)
    \leq (\eta_x(0)-t)\vee0+\tilde B_x(t).
  \end{equation}
  for any $x\in[L]$.
\end{lemma}

\begin{proof}
  The explicit construction of $(\eta(t))_{t\geq0}$ leading to equation \eqref{eq:gc} is a \emph{monotone coupling}.
  In fact if we start two copies $(\eta(t))_{t\geq0}$ and $(\eta'(t))_{t\geq0}$ of the RBB process which use the same $U(1),U(2),\dots$, such that for some $t\geq0$ and $\eta_x(t)\leq\eta'_x(t)$ $\forall x\in[L]$ then, as the function $z\mapsto(z-1)\wedge 0$ is increasing, by \eqref{eq:gc} we have $\eta_x(t+1)\leq\eta'_x(t+1)$  $\forall x\in[L]$.
  This implies that the RBB process is monotone (see \cite{Li} Definition~2.3).

  To prove equation \eqref{eq:bd} we observe that, by iterating \eqref{eq:gc} we get for any $t\geq0$ and $x\in[L]$
\begin{equation}
  \label{eq:df}
  \begin{split}
      \eta_x(t)
  &=\eta_x(0)-\sum_{s=0}^{t-1}\ind_{\{\eta_x(s)>0\}}+\sum_{s=0}^{t-1}\sum_{y=1}^L\ind_{\{\eta_y(s)>0\}}\ind_{\{U_y(s+1)=x\}}\\
  &\leq \eta_x(0)-\sum_{s=0}^{t-1}\ind_{\{\eta_x(s)>0\}}+\sum_{s=0}^{t-1}\sum_{y=1}^L\ind_{\{U_y(s+1)=x\}}.
  \end{split}
\end{equation}
Define for any $x\in[L]$ 
\begin{equation*}
  \tilde B_x(t)
  :=\sum_{s=0}^{t-1}\sum_{y=1}^L\ind_{\{U_y(s+1)=x\}}.
\end{equation*}
Because the random variables $\{U_y(s)\colon y\in[L], s\in\natural\}$ are \iid uniformly distributed on $[L]$ the random vector $\tilde B(t):=(\tilde B_1(t),\dots,\tilde B_L(t))$ has Maxwell-Boltzmann distribution with $t L$ particles and $L$ sites.
As the RBB dynamics defined by \eqref{eq:gc} removes at most one particle for any site if $\eta_x(0)>0$ then $\eta_x(s)>0$ for any $s\in\{0,\dots,\eta_x(0)-1\}$.
Thus
\begin{equation*}
  \begin{split}
      \eta_x(t)\leq\eta_x(0)- \sum_{s=0}^{\eta_x(0)\wedge t-1}\ind_{\{\eta_x(s)>0\}}+\tilde B_x(t)
  &=\eta_x(0)-\eta_x(0)\wedge t+\tilde B_x(t)\\
  &=(\eta_x(0)-t)\vee0+\tilde B_x(t).
  \end{split}
\end{equation*}
If $\eta_x(0)=0$ then by \eqref{eq:df} $\eta_x(t)\leq \tilde B_x(t)$, so for any $\eta\in\Omega$ equation~\eqref{eq:bd} holds.
\end{proof}

The next result states that if we start the RBB process from any configuration, after a fixed thermalization time, with high probability there are order $L$ empty sites.

\begin{lemma}
  \label{lem:1}
  Let $(\bar w(t))_{t\geq0}$  defined in \eqref{eq:bw}.
  There exists a constant $\epsilon_0\in(0,1)$ such that for any $\epsilon\in(0,\epsilon_0]$
  \begin{equation}
    \label{eq:1}
    \sup_{\eta\in\Omega}\PP_\eta(\bar w(t)\geq 1-\epsilon)
    \leq e^{-\epsilon L},
  \end{equation}
  for any $t\geq\lfloor 2r\rfloor\vee 1$, $L\geq2$ and $\sum_x\eta_x=rL$.
\end{lemma}

\begin{proof}
  Fix $L\geq2$.
  It is enough to show that \eqref{eq:1} holds for $t=\lfloor 2r\rfloor\vee 1$.
  In fact, assuming that it holds for $t_k:=(\lfloor 2r\rfloor\vee 1)+k$ where $k\in\integer_+$, then by the Markov property
  \begin{equation*}
    \PP_\eta(\bar w(t_{k}+1)>1-\epsilon)
    =\EE_\eta[\PP_{\eta(1)}(\bar w(t_k)>1-\epsilon)]
    \leq \sup_{\eta}\PP_\eta(\bar w(t_k)\geq 1-\epsilon)
    \leq e^{-\epsilon L},
  \end{equation*}
  and \eqref{eq:1} follows for any $t\geq t_0=\lfloor 2r\rfloor\vee 1$.
 
  We prove \eqref{eq:1} for $t=t_0$.
  As the number of particles is $rL$, there exists $V\subseteq[L]$ such that $|V|=\lfloor L/2\rfloor$ and $\eta_x\leq 2r$ for any $x\in V$.
   Then, because $(\eta_x(0)- t_0)\vee0=0$ for any $x\in V$,  by \eqref{eq:bd} $\eta_x(t_0)\leq \tilde B_x(t_0):=\tilde B_x$.
   The monotonicity of the function
   \begin{equation*}
     \Omega\ni\xi\mapsto\sum_{x\in V}\ind_{\{\xi_x>0\}}\in\real
   \end{equation*}
 implies that
 \begin{equation*}
       \bar w(t_0)
       \leq\frac{1}{L}\sum_{x\in V}\ind_{\{\eta_x(t_0)>0\}} +1-\frac{\lfloor L/2\rfloor}{L}
    \leq\frac{1}{L}\sum_{x\in V}\ind_{\{\tilde B_x>0\}} +1-\frac{\lfloor L/2\rfloor}{L}.
 \end{equation*}
  So for any $\epsilon\in(0,1/3)$ 
  \begin{equation}
    \label{eq:2}
    \begin{split}
          \PP_\eta(\bar w(t_0)>1-\epsilon)
    &\leq\PP_\eta\Big(\frac{1}{L}\sum_{x\in V}\ind_{\{\tilde B_x>0\}}>\frac{\lfloor L/2\rfloor}{L}-\epsilon\Big)\\
    =\PP_\eta\Big(\frac{1}{\lfloor L/2\rfloor}\sum_{x\in V}\ind_{\{\tilde B_x>0\}}>1-\frac{L\epsilon}{\lfloor L/2\rfloor}\Big)
    &\leq \PP_\eta\Big(\frac{1}{\lfloor L/2\rfloor}\sum_{x\in V}\ind_{\{\tilde B_x>0\}}>1-3\epsilon\Big).
    \end{split}
  \end{equation}
  For $\lambda>0$ we apply the exponential Chebyshev inequality (see \cite{Sh} Chapter 1 \S 7) to get
  \begin{equation*}
    \PP_\eta\Big(\frac{1}{\lfloor L/2\rfloor}\sum_{x\in V}\ind_{\{\tilde B_x>0\}}>1-3\epsilon\Big)
    \leq e^{-\lambda (1-3\epsilon) \lfloor L/2\rfloor}\EE_\eta\Big(\prod_{x\in V}e^{\lambda \ind_{\{\tilde B_x>0\}}}\Big).
  \end{equation*}
  Because the Maxwell-Boltzmann distribution has the \emph{negative association} property  (see \cite{Du:Ra} Lemma~4 and Theorem~14), we have that
  \begin{equation*}
    \EE_\eta\Big(\prod_{x\in V}e^{\lambda \ind_{\{\tilde B_x>0\}}}\Big)
    \leq \prod_{x\in V}\EE_\eta\Big(e^{\lambda \ind_{\{\tilde B_x>0\}}} \Big)
    =\EE_\eta\Big(e^{\lambda \ind_{\{\tilde B_1>0\}}} \Big)^{\lfloor L/2\rfloor},
  \end{equation*}
  where $\ind_{\{\tilde B_1>0\}}$ follows the Bernoulli distribution with parameter $p:=1-(1-(1/L))^{t_0L}$.
  Thus by equation \eqref{eq:2}
    \begin{equation*}
    \PP_\eta(\bar w(t_0)>1-\epsilon)
    \leq e^{-\lambda (1-3\epsilon) \lfloor L/2\rfloor}\Big(e^{\lambda} p+(1-p)\Big)^{\lfloor L/2\rfloor}.
  \end{equation*}
  By optimizing over the constant $\lambda>0$ we obtain
      \begin{equation*}
    \PP_\eta(\bar w(t_0)>1-\epsilon)
    \leq e^{-I(\epsilon) \lfloor L/2\rfloor},
  \end{equation*}
    where, for $\epsilon<(1/3)(1/4)^{t_0}$,
  \begin{equation*}
    I(\epsilon)
    :=\max_{\lambda>0}\Big\{\lambda(1-3\epsilon)-\log\big[1+p(e^\lambda-1)\big]\Big\}, 
  \end{equation*}
  and the maximum is achieved at the point
  \begin{equation*}
     \lambda^*
     :=\log\Big[\frac{(1-3\epsilon)(1-p)}{3\epsilon p}\Big]>0.
  \end{equation*}
  By an explicit  computation
  \begin{equation*}
    \lim_{\epsilon\downarrow0}I(\epsilon)
    =\log(1/p)\geq -\log(1-(1/4)^{t_0}):=c>0,
  \end{equation*}
  So by taking $\epsilon\leq c/8$ small enough and such that $I(\epsilon)\geq c/2\geq4\epsilon$ the result follows.
\end{proof}
%

\begin{proofof}{Lemma~\ref{pro:exp}}
  The second part of the statement follows from the first one as an application of the exponential Chebyshev inequality.
  For the first one we will prove a bound on the discrete time
  derivative of the left hand side of \eqref{eq:ex}, (see
  \eqref{eq:gk} below).
  Then the results follows from standard arguments.
  Let $t_0= \lfloor2r\rfloor\vee 1$ then, by \eqref{eq:bd}, for any $t\leq t_0$ we have $\eta_x(t)\leq\eta_x(0)+\tilde B_x(t)$, where $\tilde B_x(t)$ is a binomial random variable with parameters $tL$ and $1/L$.
  Thus for any $L\geq2$, $t\leq t_0$ and $\lambda>0$
   \begin{equation}
     \label{eq:st}
     \EE_\eta\big[e^{\lambda\eta_x(t)}\big]
     \leq e^{\lambda\eta_x}\Big(1+\frac{e^\lambda-1}{L}\Big)^{t_0 L}
     \leq e^{\lambda\eta_x}e^{t_0(e^\lambda-1)}.
   \end{equation}
   Let $P$ be the transition matrix of $(\eta(t))_{t\geq0}$ and,
   for $\lambda>0$, define the function $\varphi_\lambda\colon\Omega\to\real$ as $\varphi_\lambda(\eta):=e^{\lambda\eta_x}$.
   For $t\geq t_0$, the result follows if we can show that there are positive constants $\theta$, $\gamma$ and $c$ such that (using the standard identification of functions with column vectors)
   \begin{equation}
     \label{eq:gk}
     P^{t_0+1}\varphi_\theta-P^{t_0}\varphi_\theta
     \leq -\gamma P^{t_0}\varphi_\theta +c.
   \end{equation}
  In fact, if we apply $P$ to both sides of \eqref{eq:gk} and iterate, we get
    \begin{equation}
     \label{eq:gkt}
     P^{t+1}\varphi_\theta
     \leq (1-\gamma) P^{t}\varphi_\theta +c
   \end{equation}
    for any $t\geq t_0$.
    Without loss of generality we assume $\gamma\in(0,1)$.
    Iterating \eqref{eq:gkt} 
   we get
   \begin{equation*}
     P^{t+1}\varphi_\theta
     \leq(1-\gamma)^{t-t_0+1}P^{t_0}\varphi_\theta+c\sum_{n=0}^{t-t_0}(1-\gamma)^n
     \leq(1-\gamma)^{t-t_0+1}P^{t_0}\varphi_\theta+\frac{c}{\gamma}
   \end{equation*}
   and
   \begin{equation}
     \label{eq:lt}
     \EE_\eta\big(e^{\theta\eta_x(t)}\big)
     \leq(1-\gamma)^{t-t_0}\EE_\eta\big(e^{\theta\eta_x(t_0)}\big)+\frac{c}{\gamma}
   \end{equation}
   for any $t>t_0$.
   By using \eqref{eq:st}, \eqref{eq:lt} and choosing the correct constants $\theta$, $\kappa$ and $\alpha$ equation \eqref{eq:ex} follows.

   To prove the bound \eqref{eq:gk} let $\epsilon_0$ be the constant in the statement of Lemma~\ref{lem:1}, and define the events
  \begin{equation*}
    E:=\big\{\bar w(t_0)\leq 1-\epsilon_0\big\}
    \qquad\text{and}\qquad
    F:=\big\{\eta_x(t_0)>0\big\}.
  \end{equation*}
  Then
    \begin{equation}
      \label{eq:AB}
      \begin{split}
           & (P^{t_0+1}\varphi_\lambda)(\eta)-(P^{t_0}\varphi_\lambda)(\eta)
    =\EE_\eta\Big[(P\varphi_\lambda)(\eta(t_0))-\varphi_\lambda(\eta(t_0))\Big]\\
    &=\EE_\eta\Big[\ind_{E\cap F}\big\{(P\varphi_\lambda)(\eta(t_0))-\varphi_\lambda(\eta(t_0))\big\}\Big]
    +\EE_\eta\Big[\ind_{(E\cap F)^c}\big\{(P\varphi_\lambda)(\eta(t_0))-\varphi_\lambda(\eta(t_0))\big\}\Big].
      \end{split}
  \end{equation}
  Observe that
  \begin{equation}
    \label{eq:gk2}
    (P\varphi_\lambda)(\eta)-\varphi_\lambda(\eta)
    =\Big(e^{-\lambda\ind_{\{\eta_x>0\}}}\Big(1+\frac{e^{\lambda}-1}{L}\Big)^{L\bar w(\eta)}-1\Big) \varphi_\lambda(\eta),
  \end{equation}
  thus the first term on the right hand side of \eqref{eq:AB} can be bounded above by
  \begin{equation}
    \label{eq:bI}
    \Big(e^{-\lambda}\Big(1+\frac{e^{\lambda}-1}{L}\Big)^{L(1-\epsilon_0)}-1\Big)\EE_\eta \big[(1-\ind_{(E\cap F)^c})\varphi_\lambda(\eta(t_0))\big].
  \end{equation}
  To bound the second one let $\bar\lambda:=\log(1+\log2)$ and choose $\lambda\leq\bar\lambda$ so that
  \begin{equation*}
    e^{-\lambda\ind_{\{\eta_x>0\}}}\Big(1+\frac{e^{\lambda}-1}{L}\Big)^{L\bar w(\eta)}-1
    \leq\Big(1+\frac{e^{\lambda}-1}{L}\Big)^L-1
    \leq 1.
  \end{equation*}
  Thus the second term on the right hand side of \eqref{eq:AB} can be bounded above by 
  \begin{equation*}
    \EE_\eta\big[\ind_{(E\cap F)^c}\varphi_\lambda(\eta(t_0))\big].
  \end{equation*}
   Furthermore by \eqref{eq:AB} and the bound \eqref{eq:bI} we get
  \begin{equation*}
    (P^{t_0+1}\varphi)(\eta)-(P^{t_0}\varphi)(\eta)
    \leq\EE\big[\big((1-\beta_\lambda) \ind_{(E\cap F)^c}+\beta_\lambda\big)\varphi(\eta(t_0))\big],
  \end{equation*}
  where
  \begin{equation*}
    \beta_\lambda
    :=e^{-\lambda}\Big(1+\frac{e^{\lambda}-1}{L}\Big)^{L(1-\epsilon_0)}-1
    \leq 1
  \end{equation*}
  for $\lambda\leq\bar\lambda$.
    Because $\ind_{(E\cap F)^c}\leq \ind_{E^c}+\ind_{F^c}$ we obtain
  \begin{equation*}
       (P^{t_0+1}\varphi_\lambda)(\eta)-(P^{t_0}\varphi_\lambda)(\eta)
    \leq (1-\beta_\lambda)\big[\EE_\eta(\varphi_\lambda(\eta(t_0)) \ind_{E^c})+ \EE_\eta(\varphi_\lambda(\eta(t_0)) \ind_{F^c})\big]+\beta_\lambda(P^{t_0}\varphi_\lambda)(\eta).
  \end{equation*}
  Using that $\eta_x(t_0)\leq r L$ and Lemma~\ref{lem:1} we get 
  \begin{equation*}
    \EE_\eta(\varphi_\lambda(\eta(t_0)) \ind_{E^c})
    \leq e^{\lambda r L}\PP_\eta(E^c)
    \leq e^{\lambda r L}e^{-\epsilon_0 L}
  \end{equation*}
  and 
    \begin{equation*}
    \EE_\eta(\varphi_\lambda(\eta(t_0)) \ind_{F^c})
    =\PP_\eta(F^c)\leq 1.
  \end{equation*}
  Thus taking $\lambda<\epsilon_0/r$
  \begin{equation*}
    \EE_\eta(\varphi_\lambda(\eta(t_0)) \ind_{E^c})+ \EE_\eta(\varphi_\lambda(\eta(t_0)) \ind_{F^c})
    \leq 2.
  \end{equation*}
  This implies that for any $\lambda\leq\bar\lambda\wedge(\epsilon_0/r)$ 
    \begin{equation*}
    (P^{t_0+1}\varphi_\lambda)(\eta)-(P^{t_0}\varphi_\lambda)(\eta)
    \leq\beta_\lambda\EE\big[\varphi_\lambda(\eta(t_0))\big] +2(1-\beta_\lambda).
  \end{equation*}
  Observe that
  \begin{equation*}
    \beta_\lambda
    \leq\exp\big\{-\lambda+(1-\epsilon_0)(e^\lambda-1)\big\}-1
    = \exp\Big\{\lambda\Big[(1-\epsilon_0)\frac{e^\lambda-1}{\lambda}-1\Big]\Big\}-1
  \end{equation*}
  and as $(e^\lambda-1)/\lambda\downarrow 1$ for $\lambda\downarrow0$ there is a positive small enough  constant $\theta\leq\bar\lambda\wedge(\epsilon_0/r)$ such that $(1-\epsilon_0)(e^{\theta}-1)/\theta\leq 1-\epsilon_0/2$.
  We define
    \begin{equation*}
    -\gamma
    :=\beta_\theta
    \leq \exp\big\{-\frac{ \theta\epsilon_0}{2}\big\}-1<0,
  \end{equation*}
   so that \eqref{eq:gk} holds with $\gamma=-\beta_\theta$ and $c:=2(1+\gamma)$.
\end{proofof}

\section{Coalescing time}
\label{sec:ct}

To prove Lemma~\ref{lem:tel} we use the path coupling technique.
We reduce the problem of bounding the total variation distance of the distributions of two copies of the RBB process starting from different initial configurations to the problem of bounding the coalescing time of two tagged particles coupled to the RBB process.

We construct an $\Omega\times[L]\times[L]$ valued process $(\chi(t))_{t\geq0}:=((\eta(t),X(t),Y(t)))_{t\geq0}$, such that $(\eta(t))_{t\geq0}$ is an RBB process, and $(X(t))_{t\geq0}$, $(Y(t))_{t\geq0}$ are the positions of two new particles as follows.
Consider the RBB process $(\eta(t))_{t\geq0}$ with $rL-1$ particles defined in \eqref{eq:gc}.
For any $t>0$ let $U_0(t)$ be a random variable uniformly distributed in $[L]$ and such that $U_0(1),U_0(2),\dots$ are \iid and independent of $U(1),U(2),\dots$.  
For $x_0,y_0\in[L]$ define $X(0):=x_0$, $Y(0):=y_0$ and for any $t\geq0$
\begin{equation*}
  \begin{split}
    X(t+1)&:=X(t)\ind_{\{\eta_{X(t)}>0\}}+U_0(t+1) \ind_{\{\eta_{X(t)}=0\}}\\
    Y(t+1)&:=Y(t)\ind_{\{\eta_{Y(t)}>0\}}+U_0(t+1) \ind_{\{\eta_{Y(t)}=0\}}.
  \end{split}
\end{equation*}
That is the tagged balls move if and only if they are alone in their respective bins (after which coalescence is guaranteed to occur).
Notice that $(\chi(t))_{t\geq0}$ is a time homogeneous Markov chain and the $\Omega$ valued processes $(\eta^X(t))_{t\geq0}$ and $(\eta^Y(t))_{t\geq0}$ defined for any $x\in[L]$ as
\begin{equation*}
  \begin{split}
    \eta^X_x(t)&:=\eta_x(t)+ \ind_{\{X(t)=x\}}\\
    \eta^Y_x(t)&:=\eta_x(t)+ \ind_{\{Y(t)=x\}},
  \end{split}
\end{equation*}
are two coupled copies of the RBB process with $L$ sites and $rL$ particles.
The processes $(\eta^X(t))_{t\geq0}$ and $(\eta^Y(t))_{t\geq0}$ are equal except for the position of the two tagged particles until they coalesce.
Furthermore for any $s\leq t$ we have $\big\{X(s)=Y(s)\big\}\subseteq \big\{X(t)=Y(t)\big\}$;
thus if $\eta^X(s)=\eta^Y(s)$ for some $s\geq0$ then $\eta^X(t)=\eta^Y(t)$ for any $t\geq s$.
We denote by $\tau$ the \emph{coalescing time} of the two new particles, namely
\begin{equation}
  \label{eq:ct}
  \tau:=\inf\big\{t\geq0\colon X(t)=Y(t)\big\}.
\end{equation}
The next result gives an upper bound on the coalescing time in term of $\Vert\eta(0)\Vert_\infty$.  
\begin{teo}
  \label{pro:2}
  There is a positive constant $c$ such that
  \begin{equation*}
    \PP_\chi\big(\tau>\kappa(\Vert\eta\Vert_\infty\vee(\log L)^c)\big)
    \leq\frac{c}{L^2},
  \end{equation*}
  for any $\chi=(\eta,x_0,y_0)\in\Omega\times[L]\times[L]$.
\end{teo}
As the proof of Theorem~\ref{pro:2} needs some extra work we first use it to prove Lemma~\ref{lem:tel}.

\begin{proofof}{Lemma~\ref{lem:tel}}
  Recall that $P_\eta^t=\PP_\eta(\eta(t)\in\cdot)$ is the distribution of $\eta(t)$ when $\eta(0)=\eta$.
  We say that two configurations $\eta,\xi\in\Omega$ are \emph{adjacent} if there are $x,y\in[L]$ such that $\xi_x=\eta_x-1$, $\xi_y=\eta_y+1$ and $\xi_z=\eta_z$ for any $z\in[L]\setminus\{x,y\}$.
  We observe that for any $\eta,\xi\in\Omega$ there is a sequence of adjacent configurations $\eta:=\zeta_0,\zeta_1,\dots,\zeta_k:=\xi$, with $k\leq r L$ and $\max_{j\in[k]}\Vert\zeta_j\Vert_\infty\leq \Vert\eta\Vert_\infty\vee \Vert\xi\Vert_\infty$.
  By the triangle inequality
  \begin{equation}
    \label{eq:tri}
    \Vert P_\eta^t-P_\xi^t\Vert
    \leq\sum_{j=1}^k \Vert P_{\zeta_{j-1}}^t-P_{\zeta_{j}}^t\Vert.
  \end{equation}
  Thus we have to bound $\Vert P_{\zeta_{j-1}}^t-P_{\zeta_{j}}^t\Vert$ for two adjacent configurations.
  We can consider a process $(\chi(t))_{t\geq0}$, defined at the beginning of this section, starting from the initial condition $\chi_{j-1}$ such that $\eta^X(0)=\zeta_{j-1}$ and $\eta^Y(0)=\zeta_{j}$.
  Then, if $\tau$ is the coalescing time defined in \eqref{eq:ct}, by Theorem~5.4 of \cite{Le:Pe} and Theorem~\ref{pro:2}, we have
    \begin{equation*}
    \Vert P_{\zeta_{j-1}}^t-P_{\zeta_{j}}^t\Vert
    \leq \PP_{\chi_{j-1}}(\tau>t)
    \leq \PP_{\chi_{j-1}}\big(\tau>\kappa\big( (\Vert\zeta_{j-1}\Vert_\infty-1)\vee (\log L)^c\big)\big)
    \leq\frac{c}{L^2}
  \end{equation*}
  for any $t\geq \kappa\big( \Vert\eta\Vert_\infty\vee \Vert\xi\Vert_\infty\vee (\log L)^c\big)\geq \kappa\big( (\Vert\zeta_{j-1}\Vert_\infty-1)\vee (\log L)^c\big)$.
  Thus by \eqref{eq:tri} the result follows.
\end{proofof}

The proof of Theorem~\ref{pro:2} is based on Proposition~\ref{cor:2} which derives from Lemmata \ref{lem:2} and \ref{lem:3}.
Lemma~\ref{lem:2}  states that the occupation number of the sites of the tagged particles, for $t$ large enough, is unlikely to be too high.
\begin{lemma}
  \label{lem:2}
  There exist finite positive constants $c_1$ and $c_2$ such that if $\bar t:=c_1\Vert\eta\Vert_\infty$ and $\bar a=c_2\log (1+\Vert\eta\Vert_\infty)$ then
  \begin{equation*}
    \PP_\chi(\eta_{X(\bar t)}(\bar t)\vee \eta_{Y(\bar t)}(\bar t)\leq \bar a)
    \geq\frac{1}{2}
  \end{equation*}
  for any $\chi=(\eta,x_0,y_0)\in\Omega\times[L]\times[L]$.
\end{lemma}

\begin{proof}
 To prove this lemma we decouple the (possible) path of the tagged particles from the environment process $(\eta(t))_{t\geq0}$ and \eqref{eq:r1}.
  Then the result follows using the exponential bound of Lemma~\ref{pro:exp}.
  
  For any $t\geq0$ we have
  \begin{equation*}
    X(t),Y(t)\in\{x_0,y_0\}\cup\big\{U_0(1),\dots,U_0(t)\big\}
    :=\mathcal{U}(t).
  \end{equation*}
  Using a union bound, independence of $\mathcal{U}(t)$ and $\eta(t)$, and a crude upper bound on $\vert\mathcal{U}(t)\vert$ for any $a>0$ we get 
  \begin{equation}
    \label{eq:r1}
    \begin{split}
      \PP_\chi(\eta_{X(t)}\vee \eta_{Y(t)}>a)
      =\PP_\chi(\exists\, z\in \mathcal{U}(t)\colon \eta_z(t)>a)
          \leq (2+t) \max_{z\in[L]}\PP_\chi(\eta_z(t)>a).
    \end{split}
  \end{equation}
  Furthermore using \eqref{eq:Ce} we have
  \begin{equation*}
          \PP_\chi(\eta_{X(t)}\vee \eta_{Y(t)}>a)
      \leq (2+t)2\kappa\exp\big\{-\theta(a-(\Vert\eta\Vert_\infty-\alpha t)\vee 0)\big\}.
  \end{equation*}
  The result follows if one can choose $a$ and $t$ such that the last term in the above inequality is less than $1/2$.
  Taking $c_1:=\lfloor1/\alpha\rfloor+1$, $t=\bar t:=c_1\Vert\eta\Vert_\infty$ and $c_2$ such that
  \begin{equation*}
    a=\bar a:=
    c_2\log(1+\Vert\eta\Vert_\infty)
    \geq \frac{1}{\theta}\log\Big[4 \kappa\Big(\Big\lfloor\frac{r\Vert\eta\Vert_\infty}{\alpha}\Big\rfloor+3\Big) \Big],
  \end{equation*}
  this happens.
\end{proof}

The next lemma links the coalescing time with the starting site occupation numbers of the tagged particles.

\begin{lemma}
  There exist a positive constant $c$ such that, for any $L\geq c$ and $\chi=(\eta,x_0,y_0)\in\Omega\times[L]\times[L]$,
  \label{lem:3}
  \begin{equation*}
    \PP_\chi(\tau\leq \eta_{x_0}\vee \eta_{y_0}+\lfloor 2r\rfloor+1)
    \geq \Big(\frac{1}{c}\Big)^{\eta_{x_0}\vee \eta_{y_0}+\lfloor 2r\rfloor}.
  \end{equation*}
\end{lemma}

\begin{proof}
  Define $\bar t:=\eta_{x_0}\vee \eta_{y_0}+\lfloor 2r\rfloor$ and for any $\xi\in\Omega$ let $W(\xi):=\{x\in[L]\colon \xi_x=0\}$ be set of the empty sites of the configuration $\xi$.
  As $\tau\leq\inf\big\{t\geq0\colon \eta_{X(t)}=\eta_{Y(t)}=0\big\}+1$, we have that for any $\epsilon>0$
  \begin{equation*}
       \big\{\tau\leq\bar t+1\big\}
       \supseteq \big\{\tilde{B}_{x_0}(\bar t)=\tilde{B}_{y_0}(\bar t)=0\big\}\cap\bigcap_{s=1}^{\bar t}\big\{U_0(s)\in W(\eta(\bar t)), \vert W(\eta(\bar t))\vert>\epsilon L \big\}.
     \end{equation*}
     The first event implies that at time $\bar t$ the sites $x_0$ and $y_0$ may be occupied only by the tagged particles and the second one implies at time $\bar t$ the sites occupied by the tagged particles, if different from $x_0$ and $y_0$, are empty.
  Thus $\PP_\chi(\tau\leq\bar t+1)$ is bounded from below by
  \begin{equation}
    \label{eq:b1}
    \PP_\chi\big(\tilde{B}_{x_0}(\bar t)=\tilde{B}_{y_0}(\bar t)=0\big)
    \PP_\chi\Big(\bigcap_{s=1}^{\bar t}\big\{U_0(s)\in W(\eta(\bar t)), \vert W(\eta(\bar t))\vert>\epsilon L\Big\vert \tilde{B}_{x_0}(\bar t)=\tilde{B}_{y_0}(\bar t)=0\Big).
  \end{equation}
  For the first term, as $\tilde{B}(\bar t)$ has Maxwell-Boltzmann distribution with $\bar t L$ particles and $L$ sites (see Lemma~\ref{lem:dom}), we get
  \begin{equation}
    \label{eq:3b}
    \PP_\chi\big(\tilde{B}_{x_0}(\bar t)=\tilde{B}_{y_0}(\bar t)=0\big)
    =\Big(1-\frac{2}{L}\Big)^{\bar t L}
    \geq \Big(\frac{1}{16}\Big)^{\bar t},
  \end{equation}
  for any  $L\geq 4$.
  To bound the second factor in \eqref{eq:b1} we introduce an $\Omega\times[L]\times[L]$ valued process $(\chi'(t))_{t\geq0}:=((\eta'(t),X'(t),Y'(t)))_{t\geq0}$ such that
  \begin{equation*}
    \PP_\chi\big(\chi'(1)\in \Gamma_1,\dots, \chi'(t)\in \Gamma_t\big)
    =\PP_\chi\big(\chi(1)\in \Gamma_1,\dots, \chi(t)\in \Gamma_t\big\vert \tilde{B}_{x_0}(\bar t)=\tilde{B}_{y_0}(\bar t)=0\big)
  \end{equation*}
  for any $t\leq\bar t$ and $\Gamma_1,\dots,\Gamma_t\subseteq\Omega\times[L]\times[L]$.
  The process $(\chi'(t))_{t\geq0}$ for $t\leq\bar t$ has the same distribution of $(\chi(t))_{t\geq0}$ conditioned to  $\tilde{B}_{x_0}(\bar t)=\tilde{B}_{y_0}(\bar t)=0$.
  More precisely $\chi'(0):=\chi$ and $(\eta'(t))_{t\geq0}$ is the Markov chain recursively defined by
  \begin{equation*}
  \eta_x'(t+1)
  :=\eta_x'(t)-\ind_{\{\eta_x'(t)>0\}}+\sum_{y=1}^L\ind_{\{\eta_y'(t)>0\}}\ind_{\{U_y'(t+1)=x\}},
  \qquad x\in[L],
\end{equation*}
where  $U'(t):=(U_1'(t),\dots,U_L'(t))$, $U_1'(t),\dots,U_L'(t)$ are \iid random variables uniformly distributed on $[L]\setminus\{x_0,y_0\}$ and such that $U'(1),U'(2),\dots$ are independent and independent of $U_0(1),U_0(2),\dots$.  
Furthermore $(X'(t))_{t\geq0}$ and $(Y'(t))_{t\geq0}$ are recursively defined by
\begin{equation*}
  \begin{split}
    X'(t+1)&:=X'(t)\ind_{\{\eta_{X'(t)}>0\}}+U_0(t+1) \ind_{\{\eta_{X'(t)}=0\}}\\
    Y'(t+1)&:=Y'(t)\ind_{\{\eta_{Y'(t)}>0\}}+U_0(t+1) \ind_{\{\eta_{Y'(t)}=0\}}.
  \end{split}
\end{equation*}
Thus
  \begin{equation*}
    \begin{split}
      \PP_\chi\Big(\bigcap_{s=1}^{\bar t}\big\{U_0(s)\in W(\eta(\bar t)), \vert W(\eta(\bar t))\vert>\epsilon L\big\}\Big\vert \tilde{B}_{x_0}(\bar t)=\tilde{B}_{y_0}(\bar t)=0\Big)\\
      =      \PP_\chi\Big(\bigcap_{s=1}^{\bar t}\big\{U_0(s)\in W(\eta'(\bar t)), \vert W(\eta'(\bar t))\vert>\epsilon L, \big\}\Big)\\
       =\sum_{w>\epsilon L}\PP_\chi\Big(\bigcap_{s=1}^{\bar t}\big\{ U_0(s)\in W(\eta'(\bar t))\}\Big\vert \vert W(\eta'(\bar t))\vert=w\Big) \PP_\chi\big(\vert W(\eta'(\bar t))\vert=w\big).
    \end{split}
  \end{equation*}
  By the independence of $U_0(1),\dots,U_0(\bar t),\eta'(\bar t)$ we get
  \begin{equation*}
       \PP_\chi\Big(\bigcap_{s=1}^{\bar t}\big\{ U_0(s)\in W(\eta'(\bar t))\}\Big\vert \vert W(\eta'(\bar t))\vert=w\Big)
      =\Big(\frac{w}{L}\Big)^{\bar t}
      > \epsilon^{\bar t}.
  \end{equation*}
  for $w>\epsilon L$.
  Thus
  \begin{equation}
    \label{eq:2b}
    \begin{split}
             \PP_\chi\Big(\bigcap_{s=1}^{\bar t}\big\{U_0(s)\in W(\eta'(\bar t)), \vert W(\eta'(\bar t))\vert>\epsilon L \big\}\Big)
       \geq \epsilon^{\bar t}\PP_\chi\big(\vert W(\eta'(\bar t))\vert>\epsilon L\big)\\
       = \epsilon^{\bar t}\,\Big(1-\PP_\chi\Big(\frac{1}{L}\sum_{x=1}^L\ind_{\{\eta'_x(\bar t)>0\}}\geq1-\epsilon \Big)\Big).
    \end{split}
  \end{equation}
  We claim that there exists a constant $\epsilon_0'\in(0,1)$ such that for any $\epsilon\in(0,\epsilon_0']$
  \begin{equation}
    \label{eq:1b}
    \sup_{\eta\in\Omega}\PP_\eta\Big(\frac{1}{L}\sum_{x=1}^L\ind_{\{\eta'_x(t)>0\}}\geq1-\epsilon \Big)
    \leq e^{-\epsilon L},
  \end{equation}
  for any $t\geq\lfloor 2r\rfloor\vee 1$, $L\geq4$ and $\sum_x\eta_x=rL-1$.
  This result is the analogue of Lemma~\ref{lem:1} for the process $(\eta'(t))_{t\geq0}$ and can be proved along the same lines.
  We briefly sketch the proof.
  As in the proof of Lemma~\ref{lem:1} it is enough to prove \eqref{eq:1b} for $t=t_0=\lfloor 2r \rfloor\vee 1$ and following the same arguments we can show that
  \begin{equation*}
    \PP_\eta\Big(\frac{1}{L}\sum_{x=1}^L\ind_{\{\eta'_x(t_0)>0\}}>1-\epsilon\Big)
    \leq \PP_\eta\Big(\frac{1}{\lfloor L/2\rfloor}\sum_{x\in V}\ind_{\{\tilde B'_x>0\}}>1-3\epsilon\Big),
  \end{equation*}
  where $(\tilde B'_x)_{x\in[L]\setminus\{x_0,y_0\}}$ is Maxwell-Boltzmann distributed with $t_0 L$ particles and $L-2$ sites.
  Using again exponential Chebyshev inequality and negative association of Maxwell-Boltzmann distribution we get for any $\lambda>0$
   \begin{equation*}
    \PP_\eta\Big(\frac{1}{L}\sum_{x=1}^L\ind_{\{\eta'_x(t_0)>0\}}>1-\epsilon\Big)
    \leq e^{-\lambda (1-3\epsilon) \lfloor L/2\rfloor}\Big(e^{\lambda} p'+(1-p')\Big)^{\lfloor L/2\rfloor-2}.
  \end{equation*}
  where $p'=1-(1-1/(L-2))^{t_0L}$.
  The rest of the proof is exactly the same of Lemma~\ref{lem:1} and \eqref{eq:1b} is proved.
  By inequalities \eqref{eq:b1}, \eqref{eq:3b}, \eqref{eq:2b} and \eqref{eq:1b} we get
  \begin{equation*}
    \PP(\tau\leq\bar t+1)
    \geq \Big(\frac{\epsilon_0'}{16}\Big)^{\bar t}(1-e^{-\epsilon_0' L}).
  \end{equation*}
  Taking $L$ large enough the result follows.   
\end{proof}

From the last two lemmata we can prove the next statement

\begin{pro}
  \label{cor:2}
  There is a positive constant $\beta$ such that for any $\chi=(\eta,x_0,y_0)\in\Omega\times[L]\times[L]$
  \begin{equation*}
    \PP_\chi(\tau\leq\beta\Vert\eta\Vert_\infty)\geq(1+\Vert\eta\Vert_\infty)^{-\beta}.
  \end{equation*}
\end{pro}

\begin{proof}
  Let $t:=c_1\Vert\eta\Vert_\infty$, $a:=c_2\log(1+\Vert\eta\Vert_\infty)$ and $h:=\eta_{x_0}\vee \eta_{y_0}+\lfloor 2r\rfloor+1$, with $c_1$ and $c_2$ as in Lemma~\ref{lem:2}.
  Then by the Markov property
  \begin{equation}
    \label{eq:co2:1}
    \begin{split}
          \PP_\chi(\tau\leq t+h)
    \geq \PP_\chi(\tau\leq t+h,\eta_{X(t)}\vee \eta_{Y(t)}\leq a)\\
    = \EE_\chi\big[\ind_{\{\eta_{X(t)}\vee \eta_{Y(t)}\leq a\}}\PP_{\chi(t)}(\tau\leq h)\big].
    \end{split}
  \end{equation}
  By Lemma~\ref{lem:3} we have that
  \begin{equation*}
    \begin{split}
      \ind_{\{\eta_{X(t)}\vee \eta_{Y(t)}\leq a\}}\PP_{\chi(t)}(\tau\leq h)
      &\geq\ind_{\{\eta_{X(t)}\vee \eta_{Y(t)}\leq a\}}\Big(\frac{1}{c}\Big)^{\eta_{X(t)}\vee \eta_{Y(t)}+\lfloor 2r\rfloor}\\
      &\geq\ind_{\{\eta_{X(t)}\vee \eta_{Y(t)}\leq a\}} \Big(\frac{1}{c}\Big)^{a+\lfloor 2r\rfloor}.
    \end{split}
  \end{equation*}
  By plugging this bound into \eqref{eq:co2:1} and using Lemma~\ref{lem:2} we get
  \begin{equation*}
      \PP_{\chi}(\tau\leq t+h)
      \geq \frac{1}{2}\Big(\frac{1}{c}\Big)^{a+\lfloor 2r\rfloor}
      = \frac{1}{2}\Big(\frac{1}{c}\Big)^{\lfloor 2r\rfloor}\frac{1}{(1+\Vert\eta\Vert_\infty)^{c_2\log c}}.
  \end{equation*}
  From this the result follows.
\end{proof}

We finally are in a position to prove Theorem~\ref{pro:2} 

\begin{proofof}{Theorem~\ref{pro:2}}
  Let $a>0$ to be chosen later.
  To bound the tail distribution of $\tau$ we consider first the case in which the tagged particles did not coalesce and $\Vert\eta(s)\Vert_\infty\leq a$ for any $s\leq t$.
  Then the case in which $\Vert\eta(s)\Vert_\infty> a$ for some $s\leq t$, see the first line of \eqref{eq:sp}.
  Finally we bound the probability of these events using Proposition~\ref{cor:2}.
  
  Let $\beta\geq1$ such that the statement of Proposition~\ref{cor:2} holds and let $t>0$ and $a\geq 1$ two parameters we will adjust later.
  Consider the decreasing sequence of events $E_1,E_2,\dots$ defined by
  \begin{equation*}
    t_k
    :=\lfloor t+2k a\beta\rfloor,
    \ k\in\integer_+,
    \qquad
    F_{h-1}:=\bigcap_{k=0}^{h-1}\big\{\Vert\eta(t_k)\Vert_\infty\leq a\big\},
    \qquad
    E_h
    :=\big\{\tau>t_h\big\}\cap F_{h-1}.
  \end{equation*}
  By the Markov property
  \begin{equation*}
    \begin{split}
          \PP_\chi(E_{h+1})
    &=\PP_\chi(E_{h+1}\cap \{\tau>t_{h}\})
    =\EE_\chi\big[\ind_{\{\tau>t_{h}\}\cap F_h}\PP_{\chi(t_h)}(\tau>t_{h+1}-t_h)\big]\\
   &\leq \EE_\chi\big[\ind_{\{\tau>t_{h}\}\cap F_h}\PP_{\chi(t_h)}(\tau>2a\beta-1)\big]
   \leq \EE_\chi\big[\ind_{\{\tau>t_{h}\}\cap F_h}\PP_{\chi(t_h)}(\tau>a\beta)\big].
    \end{split}
  \end{equation*}
  As $F_h$ implies $\Vert\eta(t_h)\Vert_\infty\leq a$, by Proposition~\ref{cor:2}, we have $\PP_{\chi(t_h)}(\tau>a\beta)\leq 1-(1+a)^{-\beta}$ so that
  \begin{equation*}
    \PP_\chi(E_{h+1})
    \leq \PP_\chi(F_{h}\cap \{\tau>t_{h}\})\big(1-(1+a)^{-\beta}\big) 
    \leq \PP_\chi(E_{h})\exp\big\{-(1+a)^{-\beta}\big\}.
  \end{equation*}
  Iterating we get 
  \begin{equation*}
    \PP_\chi(E_{h})
    \leq\exp\big\{-h(1+a)^{-\beta}\big\}
  \end{equation*}  
  for any $h\in\natural$.
  We then have
  \begin{equation}
    \label{eq:sp}
    \begin{split}
          \PP_\chi(\tau>t_h)
    &=\PP_\chi(\{\tau>t_h\}\cap F_h)+\PP_\chi(\{\tau>t_h\}\cap F_h^c)
    \leq\PP_\chi(E_h)+\PP_\chi(F_h^c)\\
    &\leq\exp\big\{-h(1+a)^{-\beta}\big\}+h\sup_{u\geq t}\PP_\chi(\Vert\eta(u)\Vert_\infty>a).
    \end{split}
  \end{equation}
  Observe that for any $u>0$
  \begin{equation}
    \label{eq:bk}
    \PP_\chi(\Vert\eta(u)\Vert_\infty>a)
    =\PP_\chi(\exists x\in[L]\colon \eta_x(u)>a)
    \leq L \sup_{x\in[L]}\PP_\chi(\eta_x(u)>a).
  \end{equation}
  Now we choose $\bar u:=\Vert\eta\Vert_\infty/\alpha$, where $\alpha$ is the positive constant appearing in \eqref{eq:Ce},
  so that for any $u\geq \bar u$ and $x\in[L]$ we have $(\eta_x-\alpha u)\wedge0=0$.
  By \eqref{eq:Ce}
  \begin{equation*}
    \PP_\chi(\eta_x(u)>a)
    \leq\PP_\chi(\eta_x(u)\geq a)
    \leq 2\kappa e^{-\theta a}.
  \end{equation*}
  Plugging this bound into \eqref{eq:bk} we have
  \begin{equation*}
    \PP_\chi(\Vert\eta(u)\Vert_\infty>a)
    \leq2\kappa Le^{-\theta a}
  \end{equation*}
  for any $u\geq \bar u$.
  By \eqref{eq:sp} we get
    \begin{equation*}
          \PP_\chi(\tau>t_h)
     \leq \exp\big\{-h(1+a)^{-\beta}\big\}+2\kappa hLe^{-\theta a}.
  \end{equation*}
  Taking $h=\lfloor(\log L)^{2+\beta}\rfloor+1$, $a:=(4/\theta)\log L$ and $L$ large enough to have $a\geq1$, we have that
      \begin{equation*}
        \begin{split}
                    \PP_\chi(\tau>t_h)
     \leq \exp\Big\{-\Big(\frac{\log L}{1+\frac{4}{\theta}\log L}\Big)^{\beta}(\log L)^2\Big\}+\frac{(\log L)^{2+\beta}}{L^3}\\
     \leq \exp\Big\{-\Big(\frac{\theta}{4}\Big)^\beta(\log L)^2\Big\}+\frac{(\log L)^{2+\beta}}{L^3}
     =o(L^{-2}).
        \end{split}
  \end{equation*}
\end{proofof}

\subsection*{Acknowledgements}
This work has been supported by the PRIN 20155PAWZB ``Large Scale Random Structures''.
We thank the anonymous referees for their improving comments and suggestions.

\end{document}